\documentclass[10pt,reqno]{amsart}
\usepackage{amsmath,amssymb,amsthm}
\usepackage{graphicx}
\usepackage[all]{xypic}
\usepackage[utf8,nocaptions]{vietnam}
\usepackage{hyperref}
\input xypic
\def\thmsection{section}
\def\thmchangesection{changesection}

\def\thmchangechapter{changechapter}
\def\thmchange{change}
\def\thmplain{plain}
\ifx\theoremnumberstyle\thmplain
  \theoremstyle{break-italic}
  \newtheorem{satz}{Satz}
\else
  \ifx\theoremnumberstyle\thmsection
    \theoremstyle{break-italic}
    \newtheorem{satz}{Satz}[section]
  \else
    \ifx\theoremnumberstyle\thmchange
      \swapnumbers
      \theoremstyle{break-italic}
      \newtheorem{satz}{Satz}
    \else
      \ifx\theoremnumberstyle\thmchangesection
         \swapnumbers
         \theoremstyle{break-italic}
         \newtheorem{satz}{Satz}[section]
      \else
        \ifx\theoremnumberstyle\thmchangechapter
           \swapnumbers
           \theoremstyle{break-italic}
           \newtheorem{satz}{Satz}[chapter]
        \else
          \ifx\theoremnumberstyle\thmsection
             \theoremstyle{break-italic}
             \newtheorem{satz}{Satz}[section]
          \else
            \theoremstyle{break-italic}
            \newtheorem{satz}{Satz}[section]
          \fi
        \fi
      \fi
    \fi
  \fi
\fi

\theoremstyle{break-italic}
\newtheorem{theorem}[satz]{Theorem}
\newtheorem{lemma}[satz]{Lemma}

\newtheorem{corollary}[satz]{Corollary}
\newtheorem{Proposition}[satz]{Proposition}

\newtheorem*{conjecture*}{Conjecture}

\theoremstyle{break-roman}
\newtheorem{definition}[satz]{Definition}

\newtheorem{example}[satz]{Example}

\newtheorem{remark}[satz]{Remark}

\theoremstyle{standard}

\newtheorem{claim}[satz]{Claim}
\theoremstyle{varthm-roman}
\newtheorem*{varthm-roman}{}
\theoremstyle{varthm-italic}
\newtheorem*{varthm-italic}{}
\theoremstyle{varthm-roman-break}
\newtheorem*{varthm-roman-break}{}
\theoremstyle{varthm-italic-break}
\newtheorem*{varthm-italic-break}{}
\theoremstyle{varthm-roman-no-punctuation}
\newtheorem{varthm-roman-no-punctuation-numbered}[satz]{}
\theoremstyle{varthm-italic-no-punctuation}
\newtheorem{varthm-italic-no-punctuation-numbered}[satz]{}

\newenvironment{varthm-roman-numbered}[1]{
  \begin{varthm-roman-no-punctuation-numbered}
    \mbox{\rm\textbf{#1}}
  }{\end{varthm-roman-no-punctuation-numbered}}
\newenvironment{varthm-italic-numbered}[1]{
  \begin{varthm-italic-no-punctuation-numbered}
    \mbox{\rm\textbf{#1}}
  }{\end{varthm-italic-no-punctuation-numbered}}
\newenvironment{varthm-roman-break-numbered}[1]{
  \begin{varthm-roman-no-punctuation-numbered}
    \mbox{\rm\textbf{#1}\newline}
  }{\end{varthm-roman-no-punctuation-numbered}}
\newenvironment{varthm-italic-break-numbered}[1]{
  \begin{varthm-italic-no-punctuation-numbered}
    \mbox{\rm\textbf{#1}}\newline
  }{\end{varthm-italic-no-punctuation-numbered}}
\numberwithin{equation}{section}

\def\ex{\begin{example}
  }
  \def\eex{\end{example}}
\def\thr{\begin{theorem}}
\def\ethr{\end{theorem}}
\def\pro{\begin{Proposition}}
\def\epro{\end{Proposition}}
\def\coro{\begin{corollary}}
\def\ecoro{\end{corollary}}
\def\df{\begin{definition}}
\def\edf{\end{definition}}
\def\lm{\begin{lemma}}
\def\elm{\end{lemma}}
\def\pf{\begin{proof}}
\def\epf{\end{proof}}
\def\problem{\begin{problem}}
\def\eproblem{\end{problem}}

\def\it{\begin{itemize}}
\def\hit{\end{itemize}}
\def\rem{\begin{remark}}
\def\erem{\end{remark}}

\def\cla{\begin{claim}}
\def\ecla{\end{claim}}

\newcommand{\seq}[1]{\left<#1\right>}

\def\Tr{\text{Tr }}

\usepackage{chngcntr}
\everymath{\displaystyle}
\begin{document}
\title[Optimization of maximal quantum $f$-divergences between unitary orbits]{Optimization of maximal quantum
$f$-divergences between unitary orbits}
\author{Hoang Minh Nguyen, Hoang An Nguyen, Cong Trinh Le}

\date{\today}
\begin{abstract}
 Maximal quantum $f$-divergences, defined via the commutant Radon--Nikodym
derivative, form a fundamental class of distinguishability measures for quantum states associated with operator convex functions. In this paper, we study the optimization of maximal quantum $f$-divergences along unitary
orbits of two quantum states.

For any operator convex function $f:(0,+\infty)\to\mathbb{R}$, we determine 
the exact minimum and maximum of
\[
U \longmapsto \widehat S_f(\rho\|U^*\sigma U)
\]
over the unitary group, and derive explicit spectral formulas for these
extremal values together with complete characterizations of the unitaries
that attain them.

Our approach combines the integral representation of operator convex functions with majorization theory and a unitary-orbit variational method. A key step is to show that any extremizer must commute with the reference state, which reduces the noncommutative optimization problem to a spectral permutation problem. As a consequence, the minimum is achieved by pairing the decreasing eigenvalues of $\rho$ and $\sigma$, while the maximum corresponds to pairing the decreasing eigenvalues of $\rho$ with the increasing eigenvalues of $\sigma$. Hence, the range of the maximal quantum $f$-divergence along the unitary orbit is exactly the closed interval determined by these two extremal configurations.

Finally, we compare our results with recent unitary-orbit optimization results for quantum $f$-divergences defined via the quantum hockey-stick divergence, highlighting fundamental structural differences between the two
frameworks. Our findings extend earlier extremal results for Umegaki, R\'enyi, and related quantum divergences, and clarify the distinct operator-theoretic nature of maximal quantum $f$-divergences.
\end{abstract}
\maketitle
\section{Introduction}
Quantifying the dissimilarity between quantum states is a cornerstone problem in quantum information theory, underpinning diverse topics such as state discrimination, statistical inference, resource theories, and thermodynamics. Among the most effective mathematical tools for describing this dissimilarity are \emph{quantum divergences}, which extend classical information distances such as the Kullback–Leibler and Rényi divergences into the noncommutative operator setting. These quantities measure the distinguishability between quantum states and form the foundation of key operational concepts including entanglement, coherence, and mutual information~\cite{Hiai2017}.  There are some types of R\'enyi divergences which were attracted by many mathematicians, e.g. Umegaki's relative entropy, the (conventional) R\'enyi divergences, the sandwiched $\alpha$-R\'enyi divergences, the $\alpha$-$z$-R\'enyi divergences,... \cite{AD, EH, Hiai2011, Hiai2017, Mu}.

In the classical setting, Csiszár and Ali–Silvey \cite{Cs} introduced the $f$-divergence between two probability distributions $p,q$ on a finite set $X$:
\begin{equation}
S_f(p\|q) = \sum_{x \in X} q(x) f\!\left( \frac{p(x)}{q(x)} \right),
\end{equation}
where $f:(0,+\infty)\to\mathbb{R}$ is convex. The relative entropy corresponds to \( f(t) := \eta(t) := t \log t \), while the R\'enyi divergences can be expressed as (\cite{Hiai2017})
$$
D_{\alpha}(p\|q) = \frac{1}{\alpha - 1} \log \bigl( \operatorname{sign}(\alpha - 1) S_{f_{\alpha}}(p\|q) \bigr), \quad
f_{\alpha}(t) := \operatorname{sign}(\alpha - 1)t^{\alpha}.
$$

In this paper we consider the \textit{ maximal quantum f-divergences} (which is also called \textit{the quantum $f$-divergence defined through the commutant Radon–Nikodym derivative}), a framework introduced by D. Petz and M. B. Ruskai \cite{PeRus98} (see also the work of Hiai and Mosonyi \cite{Hiai2017} for many other kinds of quantum $f$-divergences and their properties). Specifically, for an operator convex function $f : (0, +\infty) \rightarrow  \mathbb R$ and for any $\rho, \sigma \in \mathbb P_n^+$, the \textit{maximal quantum $f$-divergence } of $\rho$ and $\sigma$ is defined as 
\begin{equation}
\widehat{S}_f(\rho\|\sigma):=\Tr\left[\sigma^{1/2}f(\sigma^{-1/2}\rho\sigma^{-1/2})\sigma^{1/2}\right]=\Tr\left[\sigma f\left(\sigma^{-1/2}\rho\sigma^{-1/2}\right)\right].
\end{equation}
For general $\rho$ and $\sigma$ in $\mathbb P_n$, their maximal quantum $f$-divergence is defined by taking the limitation as follows.
\begin{equation}  \label{def-maximal-đivergence-general}
\widehat{S}_f(\rho\|\sigma) = \lim_{\epsilon \downarrow 0} \widehat{S}_f(\rho+\epsilon I\|\sigma+\epsilon I),
\end{equation}
where $I$ is the identity matrix in $\mathbb M_n$. 
This definition reflects a genuine noncommutative interaction between $\rho$ and $\sigma$ via their commutant structure. The operator convex function $f$ on $(0, +\infty)$ admits the integral representation
\begin{equation}\label{opt-convex-int-rep}
f(x) = f(1) + f'(1)(x-1) + c(x-1)^2 + \!\int_{[0,+\infty)} \!\! \frac{(x-1)^2}{x+s}\,d\lambda(s), \quad x \in (0, +\infty),
\end{equation}
with $c\geq 0$ and a positive measure $\lambda$ on $[0, +\infty)$ satisfying $\!\int_{[0,+\infty)} \!\! \frac{1}{1+s}\,d\lambda(s)< +\infty$.

Since unitary transformations preserve the spectra of density matrices, two quantum states related by unitary conjugation are physically indistinguishable. The \emph{unitary orbit} of a quantum state $\rho$ is
\begin{equation}
\mathbb{U}_\rho = \{U^*\rho U \mid U \in \mathbb{U}_n\},
\end{equation}
where $\mathbb U_n$ denotes the set of unitary matrices.

 Optimization of quantum divergences $S_f(\rho\|\sigma)$ over the unitary orbits captures the extremal distinguishability achievable through spectral rearrangement:
\begin{equation} 
  \min_{V, W\in \mathbb U_n}S_f(V^*\rho V\|W^*\sigma W) 
\end{equation}
and 
\begin{equation}
 \max_{V, W\in \mathbb U_n} S_f(V^*\rho V\| W^*\sigma W).
\end{equation}

The pioneering work on these topics was carried out by Zhang and Fei (\cite{ZhF}, 2014, pertaining to Umegaki's relative entropies) as well as Yan, Yin, and Li (\cite{YYL}, 2020, concerning quantum $\alpha$-fidelities), and our recent papers (\cite{VHLD}, 2023; \cite{DLV24}, 2024; \cite{LVHD}, 2024). Recently, Li and Yan (\cite{LiYan}, 2025) studied the unitary orbit optimization of the quantum $f$-divergences $D_f(\rho\|\sigma)$ with respect to the quantum hockey-stick divergence for convex and twice differentiable functions $f : (0, +\infty) \rightarrow \mathbb R$ with $f (1) = 0$.

For any operator convex function $f : (0, +\infty) \rightarrow \mathbb R$, our main aim in this paper is to consider the following extremal problems:

\begin{equation}\label{equa:min-original}
  \min_{V, W\in \mathbb U_n}\widehat{S}_f(V^*\rho V\|W^*\sigma W) 
\end{equation}
and 
\begin{equation}\label{equa:max-original}
 \max_{V, W\in \mathbb U_n} \widehat{S}_f(V^*\rho V\| W^*\sigma W).
\end{equation}

Note that, for any unitary matrices $V$ and $W$, on account of  the unitary invariant property of $\widehat{S}_f(\rho\|\sigma)$ (\cite{Hiai2017}), we have
$$\widehat{S}_f(V^*\rho V\| W^*\sigma W)=  \widehat{S}_f(\rho\|VW^*\sigma WV^*) =  \widehat{S}_f(\rho\|U^*\sigma U),$$
	where $U=WV^* \in \mathbb U_n$. 
	
 Hence, instead of considering problems (\ref{equa:min-original}) and (\ref{equa:max-original}), in this paper  we study the following problems: 
	\begin{equation} \label{equ:min}
		\min_{U\in \mathbb U_n} \widehat{S}_f(\rho \| U^*\sigma U)
	\end{equation} 
	and 
	\begin{equation}\label{equ:max}
		\max_{U\in \mathbb U_n} \widehat{S}_f(\rho \|U^*\sigma U).
	\end{equation} 

Beyond the final spectral formulas, a key novelty of this work lies in the proof techniques developed for Claims~ \ref{claimthuong}--\ref{claimsigma}. In addition to classical tools from majorization theory and Lidskii-type inequalities \cite{Bh, MOA}, we introduce a unitary-orbit variational method combined with rearrangement principles for supermodular and submodular functions 
\cite{HLP52,Top98}. By differentiating the objective function along smooth unitary paths, we show that any extremizer must commute with the reference state $\rho$, which reduces the optimization problem to a purely
spectral permutation problem.

This approach reveals that, although maximal quantum $f$-divergences are highly nonlinear and intrinsically noncommutative, their unitary-orbit extrema are governed by a precise interplay between commutant structure and spectral rearrangement. This mechanism is fundamentally different from recent approaches based on the quantum hockey-stick divergence \cite{HirTo,LiYan}, and highlights the distinctive operator-theoretic nature of maximal quantum $f$-divergences.

The paper is organized as follows. Section \ref{subsec:majorization} collects the necessary preliminaries, including notation, basic results from majorization theory, rearrangement principles for supermodular and submodular functions, and a unitary-orbit variational framework that will be used in the proofs of the main results. Section \ref{sec:mainresults} states the main theorems on the minimization and maximization
of the maximal quantum $f$-divergence along unitary orbits and presents explicit spectral expressions for the extremal values. Section \ref{sec:proofs} is devoted to the proofs, where the optimization problem is
reduced to spectral rearrangement by combining unitary-orbit differentiation arguments with majorization and rearrangement theory. Finally, Section \ref{sec:conclusion} contains concluding remarks and a comparison with
recent results on hockey-stick-based quantum $f$-divergences due to Li and Yan \cite{LiYan}.

\section{Preliminaries}\label{subsec:majorization}

\subsection{Notation}

Throughout this paper, we use $\mathbb{M}_n$ (and similarly $\mathbb{H}_n$, $\mathbb{P}_n$, $\mathbb{P}_n^+$, $\mathbb{U}_n$, $\mathbb{D}_n$) to denote the sets of complex $n\times n$ matrices, Hermitian matrices, positive semidefinite matrices, positive definite matrices, unitary matrices, and density matrices, respectively. 

Recall that a \textit{quantum state} is represented by a density matrix $\rho  \in \mathbb{D}_n$, that is,  $\rho \in \mathbb{P}_n$ and $\operatorname{Tr}(\rho) = 1$. 

Note that, every Hermitian matrix $A \in \mathbb{H}_n$ admits the decomposition
\[
A = A_+ - A_-,
\]
where $A_+, A_- \in \mathbb{P}_n$ are the positive and negative parts of $A$ arising from its eigenvalue decomposition. 

For matrices $A, B \in \mathbb{H}_n$, we write $A \le B$ if $B - A \in \mathbb{P}_n$, that is, if $B - A$ is positive semidefinite.

For $A \in \mathbb{P}_n$, $B \in \mathbb{P}_n^+$, denote $\dfrac{A}{B}:=AB^{-1}$. If $B \in \mathbb{P}_n$, the notation $\dfrac{A}{B}$ means  $\dfrac{A}{\displaystyle\lim_{\epsilon \downarrow 0}(B+\epsilon I)}$.

For a vector $x = (x_1, \dots, x_n) \in \mathbb{R}^n$, we denote by 
\[
x_1^\downarrow \ge \cdots \ge x_n^\downarrow
\quad \text{and} \quad
x_1^\uparrow \le \cdots \le x_n^\uparrow
\]
the components of $x$ arranged in nonincreasing and nondecreasing order, respectively. The vectors $x^\downarrow$ and $x^\uparrow$ will denote $x$ with its entries ordered in descending and ascending order.

For any $A \in \mathbb{P}_n$, we denote by $\lambda(A) = (\lambda_1(A), \dots, \lambda_n(A))$ the vector of eigenvalues of $A$, with
\[
\lambda_1^\downarrow(A) \ge \cdots \ge \lambda_n^\downarrow(A)
\quad \text{and} \quad
\lambda_1^\uparrow(A) \le \cdots \le \lambda_n^\uparrow(A)
\]
representing the eigenvalues in decreasing and increasing order, respectively. We use $\lambda^\downarrow(A)$ to refer both to the ordered eigenvalue vector $(\lambda_1^\downarrow(A), \dots, \lambda_n^\downarrow(A))$ and to the corresponding diagonal matrix 
\[
\text{diag}(\lambda_1^\downarrow(A), \dots, \lambda_n^\downarrow(A)),
\]
and similarly for $\lambda^\uparrow(A)$.

For vectors $x = (x_1, \dots, x_n)$ and $y = (y_1, \dots, y_n)$ in $\mathbb{R}^n$, we write 
$$\seq{x, y} := \sum_{i=1}^n x_i y_i; \quad x \circ y := (x_1 y_1, \dots, x_n y_n).$$
For $y = (y_1, \dots, y_n)$ whose components are positive, denote
$$\dfrac{x}{y}:=x\circ y^{-1}= \big(\dfrac{x_1}{y_1},\cdots, \dfrac{x_n}{y_n}\big). $$
If components of $y = (y_1, \dots, y_n)$ are non-negative,  $\dfrac{x}{y}$ is defined by taking the limit. 

\subsection{Majorization theory for vectors in $\mathbb R^n$}
This section reviews several fundamental concepts and results concerning majorization relations between vectors in $\mathbb{R}^n$. For comprehensive expositions, see~\cite{Bh,MOA}.

\begin{definition}
Let $x = (x_1, \dots, x_n)$ and $y = (y_1, \dots, y_n)$ be vectors in $\mathbb{R}^n$. We say that $x$ is \emph{weakly majorized by} $y$, denoted $x \prec_w y$, if
\[
\sum_{i=1}^k x_i^\downarrow \le \sum_{i=1}^k y_i^\downarrow, \qquad 1 \le k \le n.
\]
If, in addition, equality holds for $k = n$, we say that $x$ is \emph{majorized by} $y$ and write $x \prec y$.
\end{definition}

The following result  establishes a relationship between the trace of products of matrices  and their eigenvalues.

\begin{lemma}[{\cite[Problem III.6.14]{Bh}}]\label{lm-trace-product-powers}
Let $A, B \in \mathbb{H}_n$. Then
\[
\seq{\lambda^{\downarrow}(A), \lambda^{\uparrow}(B)} 
\le \operatorname{Tr}\!\left(A B\right)
\le \seq{\lambda^{\downarrow}(A), \lambda^{\downarrow}(B)}.
\]
\end{lemma}

\subsection{Rearrangement principles for supermodular and submodular functions}

Let $r=(r_1,\dots,r_n)$ and $d=(d_1,\dots,d_n)$ be vectors in $(0,+\infty)^n$ such that
\[
r_1 \ge \cdots \ge r_n, \qquad d_1 \ge \cdots \ge d_n .
\]
A function $f:(0,+\infty)^2 \to \mathbb{R}$ is called \emph{supermodular} if
\[
\frac{\partial^2 f}{\partial r\,\partial d}(r,d) \ge 0,
\]
and \emph{submodular} if the above inequality is reversed. The function $f$  is called \emph{strictly supermodular} (resp. \emph{strictly submodular}) if 

\[
\frac{\partial^2 f}{\partial r\,\partial d}(r,d) > 0 \quad (\mbox{resp. } \frac{\partial^2 f}{\partial r\,\partial d}(r,d) < 0).
\]

If $f$ is strictly supermodular, then for all $d_1 \ge d_2$ and $r_1 \ge r_2$, we have (see, e.g. \cite{Top98}):
\begin{equation}
f(d_1, r_1) + f(d_2, r_2) \ge f(d_1, r_2) + f(d_2, r_1), \tag{SM}
\end{equation}
and the inequality (SM) holds strictly whenever $d_1 > d_2$ and $r_1 > r_2$:
\begin{equation*}
f(d_1, r_1) + f(d_2, r_2) > f(d_1, r_2) + f(d_2, r_1).
\end{equation*}

A fundamental consequence of supermodularity and submodularity is the
\emph{rearrangement principle} (see \cite[Ch.~10]{HLP52} and
\cite[Ch.~2]{Top98}): for any permutation $\pi$ of $\{1,\dots,n\}$,
\[
\sum_{i=1}^n f(r_i,d_{\pi(i)})
\]
is maximized when $(r_i)$ and $(d_{\pi(i)})$ are ordered in the same sense and
minimized when they are ordered in opposite senses if $f$ is supermodular,
while the opposite conclusion holds if $f$ is submodular.

These rearrangement principles are used repeatedly in the proofs of
Claims~ \ref{claimA}, \ref{claimrho}, \ref{claimsigma} to reduce unitary-orbit optimization problems to purely
spectral permutation problems.

\subsection{Unitary-orbit variational method and commutant structure}

Let $A\in\mathbb P_n^+$. Its unitary orbit is
\[
\mathbb U_A := \{U^*AU : U\in\mathcal U_n\},
\]
which is a smooth homogeneous manifold under the action of the unitary group.
Its tangent space at $A$ is given by
\[
T_A\mathbb U_A=\{[A,K]:K^*=-K\}
\]
(see, e.g. \cite[(VI.37)]{Bh}).

Accordingly, every smooth curve on $\mathbb U_A$ through $A$ can be written as
\[
A(t)=e^{-tK}Ae^{tK}, \qquad K^*=-K .
\]

Let $F$ be a Fr\'echet differentiable real-valued function defined on
$\mathbb U_A$. Stationarity of $F$ at $A$ with respect to unitary variations
means that
\[
\dfrac{d}{dt}F\!\left(e^{-tK}Ae^{tK}\right)\Big\lvert_{t=0}=0
\quad \text{for all } K^*=-K .
\]
Equivalently, the gradient $\nabla F(A)$ is orthogonal to the tangent space
$T_A\mathbb U_A$ with respect to the Hilbert--Schmidt inner product, which yields
the first-order optimality condition
\[
[A,\nabla F(A)]=0.
\]
Such commutation conditions for stationary points on unitary similarity orbits
are standard in optimization on matrix manifolds; see, e.g.
\cite[Theorem VI.4.3]{Bh}.

In Claims \ref{claimA}, \ref{claimrho} and \ref{claimsigma}, this variational method is used to show that any optimizer must commute with the reference state $\rho$. Consequently, the optimization problem reduces to a spectral rearrangement problem involving the eigenvalues of $\rho$ and $\sigma$.

\section{Main results}\label{sec:mainresults}
Let $f : (0, +\infty) \rightarrow \mathbb R$ be an operator  convex function. Then $f$ has the integral representation given by (\ref{opt-convex-int-rep}). We consider the problems (\ref{equ:min}) and (\ref{equ:max}) for the quantum $f$-divergences $\widehat{S}_f$ with respect to commutant Radon-Nikodym derivative.

Let $\rho$ and $\sigma$ be quantum states. It follows from the spectral theorem that there exist unitary matrices $V^\downarrow$ and $V^\uparrow$ such that 
\begin{equation}\label{V-lambda}
\lambda^\downarrow (\rho)= V^{\downarrow *} \rho V^\downarrow \mbox{ and } \lambda^\uparrow (\rho)= V^{\uparrow *} \rho V^\uparrow,
\end{equation}
and, there exist unitary matrices $W^\downarrow$ and $W^\uparrow$ such that 
\begin{equation}\label{W-lambda}
\lambda^\downarrow (\sigma)= W^{\downarrow *} \sigma W^\downarrow \mbox{ and } \lambda^\uparrow (\sigma)= W^{\uparrow *} \sigma W^\uparrow. 
\end{equation}
Denote 
\begin{align}\label{max-divergence-lambdadowndown}
\widehat{S}_f(\lambda^\downarrow(\rho)\|\lambda^\downarrow(\sigma))&=f(I)+c\left[\operatorname{Tr}\lambda^\downarrow(\rho)^2\lambda^\downarrow(\sigma)^{-1}-1\right]\nonumber\\
&+\int_{[0;\infty)}\operatorname{Tr}\dfrac{\lambda^\downarrow(\sigma)^{-1}\lambda^\downarrow(\rho)^{2}}{\lambda^\downarrow(\sigma)^{-1}\lambda^\downarrow(\rho)+s I} d\lambda(s)\nonumber\\
&-\int_{[0;\infty)}\-2\operatorname{Tr}\dfrac{\lambda^\downarrow(\rho)}{\lambda^\downarrow\left(\sigma\right)^{-1}\lambda^\downarrow(\rho)+sI}d\lambda(s)\nonumber\\
&+\int_{[0;\infty)}\operatorname{Tr} \dfrac{\lambda^\downarrow(\sigma)}{\lambda^\downarrow(\sigma)^{-1}\lambda^\downarrow(\rho) + sI}d\lambda(s).
\end{align}
Similarly, denote 
\begin{align}\label{max-divergence-lambdaupdown}
\widehat{S}_f(\lambda^\downarrow(\rho)\|\lambda^\uparrow(\sigma))&=f(I)+c\left[\operatorname{Tr}\lambda^\downarrow(\rho)^2\lambda^\uparrow(\sigma)^{-1}-1\right]\nonumber\\
&+\int_{[0;\infty)}\operatorname{Tr}\dfrac{\lambda^\uparrow(\sigma)^{-1}\lambda^\downarrow(\rho)^{2}}{\lambda^\uparrow(\sigma)^{-1}\lambda^\downarrow(\rho)+s I} d\lambda(s)\nonumber\\
&-\int_{[0;\infty)}\-2\operatorname{Tr}\dfrac{\lambda^\downarrow(\rho)}{\lambda^\uparrow\left(\sigma\right)^{-1}\lambda^\downarrow(\rho)+sI}d\lambda(s)\nonumber\\
&+\int_{[0;\infty)}\operatorname{Tr} \dfrac{\lambda^\uparrow(\sigma)}{\lambda^\uparrow(\sigma)^{-1}\lambda^\downarrow(\rho) + sI}d\lambda(s).
\end{align}

The following results are our main results in this paper.
\begin{theorem}\label{min-crnd} Let $\rho$ and $\sigma$ be quantum states. Then
$$ 
\min_{U\in\mathbb{U}_n} \widehat{S}_f(\rho\|U^*\sigma U)= \widehat{S}_f(\lambda^\downarrow(\rho)\|\lambda^\downarrow(\sigma));$$
$$\arg\min_{U\in\mathbb{U}_n} \widehat{S}_f(\rho\|U^*\sigma U)=W^\downarrow V^{\downarrow*}.$$
\end{theorem}

\begin{theorem}\label{max-crnd} Let $\rho$ and $\sigma$ be quantum states. Then
$$
\max_{U\in\mathbb{U}_n} \widehat{S}_f(\rho\|U^*\sigma U)= \widehat{S}_f(\lambda^\downarrow(\rho)\|\lambda^\uparrow(\sigma));$$
$$\arg\max_{U\in\mathbb{U}_n} \widehat{S}_f(\rho\|U^*\sigma U)=W^\uparrow V^{\downarrow*}.$$
\end{theorem}

\begin{corollary} \label{coro-f-diverg-interval} 		
Let  $\rho$ and $ \sigma$ be quantum states. Then,  the set $\Big\{\widehat{S}_f(\rho\|U^*\sigma U), U\in \mathbb U_n\Big\}$ is exactly the interval $\Big[\widehat{S}_f(\lambda^\downarrow(\rho)\|\lambda^\downarrow(\sigma)), \widehat{S}_f(\lambda^\downarrow(\rho)\|\lambda^\uparrow(\sigma))\Big]$.
	\end{corollary}
	\begin{proof}
Note that the map  $U\in \mathbb U_n \longmapsto (U^*\sigma U)$ is continuous. Moreover, it follows from the continuity of the function $f$ and the linearity of the trace function that the map $U\in \mathbb U_n \longmapsto  \operatorname{Tr}[(U^*\sigma U) f((U^*\sigma U)^{-1/2}\rho (U^*\sigma U)^{-1/2})]$ is continuous. Then we get the continuity of the function $U\in \mathbb U_n \longmapsto \widehat{S}_f(\rho\|U^*\sigma U)$. 
  
On the other hand,  it is well-known that the unitary orbit $\mathbb{U}_\sigma$ is connected and that the image of a connected set under a continuous map is also connected \cite[Theorem 4.22]{Ru}. It follows that the set $\Big\{\widehat{S}_f(\rho\|U^*\sigma U): U\in \mathbb U_n\Big\}$ is connected and therefore fills out the interval between the minimum and maximum values obtained in Theorem \ref{min-crnd} and Theorem \ref{max-crnd}.     
	\end{proof}
	
\section{Proof of main results}\label{sec:proofs}
By definition, 
\begin{equation}
\widehat{S}_f(\rho\|\sigma):=\operatorname{Tr}\left[\sigma^{1/2}f(\sigma^{-1/2}\rho\sigma^{-1/2})\sigma^{1/2}\right]=\operatorname{Tr}\left[\sigma f\left(\sigma^{-1/2}\rho\sigma^{-1/2}\right)\right].
\end{equation}
Using the integral representation (\ref{opt-convex-int-rep}) of the operator convex funtion $f$, we have an explicit representation for $f(\sigma^{-1/2}\rho\sigma^{-1/2})$.

\begin{align*}
f\left(\sigma^{-1/2}\rho\sigma^{-1/2}\right)&=f(I)+f'(I)\left(\sigma^{-1/2}\rho\sigma^{-1/2}-I\right)+c\left(\sigma^{-1/2}\rho\sigma^{-1/2}-I\right)^2\\
&+\int_{[0;+\infty)} \dfrac{\left(\sigma^{-1/2}\rho\sigma^{-1/2}-I\right)^2}{\sigma^{-1/2}\rho\sigma^{-1/2}+sI} d\lambda(s).
\end{align*}
It implies that
\begin{align*}
\sigma f\left(\sigma^{-1/2}\rho\sigma^{-1/2}\right)&=f(I)\sigma+\sigma f'(I)\left(\sigma^{-1/2}\rho\sigma^{-1/2}-I\right)+c\sigma\left(\sigma^{-1/2}\rho\sigma^{-1/2}-I\right)^2\\
&+\int_{[0;+\infty)} \dfrac{\sigma\left(\sigma^{-1/2}\rho\sigma^{-1/2}-I\right)^2}{\sigma^{-1/2}\rho\sigma^{-1/2}+sI} d\lambda(s).
\end{align*}
By linearity of the  trace function, we have
\begin{align*}
\operatorname{Tr}\sigma f\left(\sigma^{-1/2}\rho\sigma^{-1/2}\right)&=f(I)\operatorname{Tr}\sigma+f'(I)\operatorname{Tr}\sigma \left(\sigma^{-1/2}\rho\sigma^{-1/2}-I\right)\\
&\quad+c\operatorname{Tr}\sigma\left(\sigma^{-1/2}\rho\sigma^{-1/2}-I\right)^2\\
&\quad+\int_{[0;+\infty)}\operatorname{Tr} \dfrac{\sigma\left(\sigma^{-1/2}\rho\sigma^{-1/2}-I\right)^2}{\sigma^{-1/2}\rho\sigma^{-1/2}+sI} d\lambda(s)\\
&=f(I)\operatorname{Tr}\sigma+f'(I)\operatorname{Tr}\sigma \left(\sigma^{-1/2}\rho\sigma^{-1/2}-I\right)\\
&\quad+c\operatorname{Tr}\sigma\left(\sigma^{-1/2}\rho\sigma^{-1/2}-I\right)^2\\
&\quad+\int_{[0;+\infty)}\operatorname{Tr} \dfrac{\sigma^{1/2}\left(\sigma^{-1/2}\rho\sigma^{-1/2}-I\right)^2\sigma^{1/2}}{\sigma^{-1/2}\rho\sigma^{-1/2}+sI} d\lambda(s).
\end{align*}
Note that 
\begin{equation} \label{claim1}
\operatorname{Tr}\sigma \left(\sigma^{-1/2}\rho\sigma^{-1/2}-I\right)=0; \mbox{and } 
\operatorname{Tr}\sigma \left(\sigma^{-1/2}\rho\sigma^{-1/2}-I\right)^2=\operatorname{Tr}\left(\sigma^{-1}\rho^2\right)-1.
\end{equation}
In fact, we have 
$$
\operatorname{Tr}\sigma \left(\sigma^{-1/2}\rho\sigma^{-1/2}-I\right)=\operatorname{Tr} \left(\sigma^{1/2}\rho\sigma^{-1/2}-\sigma\right)=\operatorname{Tr}(\rho)-\operatorname{Tr}(\sigma)=0.
$$
Moreover,
\begin{align*}
\left(\sigma^{-1/2}\rho\sigma^{-1/2}-I\right)^2 &= \left(\sigma^{-1/2}\rho\sigma^{-1/2}\right)^2 - 2\left(\sigma^{-1/2}\rho\sigma^{-1/2}\right) + I\\
&= \sigma^{-1/2}\rho\sigma^{-1}\rho\sigma^{-1/2} - 2\left(\sigma^{-1/2}\rho\sigma^{-1/2}\right) + I,
\end{align*}
which implies
\begin{align*}
\sigma  \left(\sigma^{-1/2}\rho\sigma^{-1/2}-I\right)^2 &= \sigma^{1/2} \rho\sigma^{-1}\rho\sigma^{-1/2} - 2\left(\sigma^{1/2}\rho\sigma^{-1/2}\right) + \sigma.
\end{align*}
Applying the trace function, we obtain
\begin{align*}
\operatorname{Tr} \sigma  \left(\sigma^{-1/2}\rho\sigma^{-1/2}-I\right)^2&= \operatorname{Tr} \left[\sigma^{1/2} \rho\sigma^{-1}\rho\sigma^{-1/2} - 2\left(\sigma^{1/2}\rho\sigma^{-1/2}\right) + \sigma\right]\\
&=\operatorname{Tr}\left(\sigma^{1/2} \rho\sigma^{-1}\rho\sigma^{-1/2}\right) - 2\operatorname{Tr}\left(\sigma^{1/2}\rho\sigma^{-1/2}\right) + \operatorname{Tr}\sigma\\
&=\operatorname{Tr}\left(\rho\sigma^{-1}\rho\right) -1 = \operatorname{Tr}\left(\sigma^{-1}\rho^2\right)-1.
\end{align*}
On the other hand, we have
\begin{align}\label{claim2}
\operatorname{Tr}\dfrac{\sigma^{1/2}\left(\sigma^{-1/2}\rho\sigma^{-1/2}-I\right)^2\sigma^{1/2}}{\sigma^{-1/2}\rho\sigma^{-1/2}+sI} &= \operatorname{Tr}\dfrac{\rho\sigma^{-1}\rho-2\rho+\sigma}{\sigma^{-1/2}\rho\sigma^{-1/2}+sI}\nonumber\\
&=\operatorname{Tr}\dfrac{\sigma^{-1}\rho^2}{\sigma^{-1/2}\rho\sigma^{-1/2}+sI}-2\operatorname{Tr}\dfrac{\rho}{\sigma^{-1/2}\rho\sigma^{-1/2}+sI}\nonumber\\
&\quad+\operatorname{Tr}\dfrac{\sigma}{\sigma^{-1/2}\rho\sigma^{-1/2}+sI}.
\end{align}
It follows from (\ref{claim1}) and (\ref{claim2}) that 
\begin{align}\label{main-repr-maximal-divergence}
\widehat{S}_f(\rho\|\sigma)&=\operatorname{Tr}\left[\sigma f\left(\sigma^{-1/2}\rho\sigma^{-1/2}\right)\right]\nonumber\\
&=f(I)+c\left[\operatorname{Tr}\left(\sigma^{-1}\rho^2\right)-1\right]+\int_{[0;\infty)} \operatorname{Tr}\dfrac{\sigma^{-1}\rho^2}{\sigma^{-1/2}\rho\sigma^{-1/2}+sI}d\lambda(s)\\
&-2\int_{[0;\infty)}\operatorname{Tr}\dfrac{\rho}{\sigma^{-1/2}\rho\sigma^{-1/2}+sI}d\lambda(s)  +\int_{[0;\infty)}\operatorname{Tr}\dfrac{\sigma}{\sigma^{-1/2}\rho\sigma^{-1/2}+sI}d\lambda(s). \nonumber
\end{align}

To optimize the maximal quantum $f$-divergence (\ref{main-repr-maximal-divergence}) between unitary orbits, we consider separately the optimization of each component in the representation of $\widehat{S}_f(\rho\|\sigma)$.

\begin{claim}\label{claimthuong}
\begin{align*}
\max_{U\in\mathbb U_n}\operatorname{Tr}\left(\rho^2(U^*\sigma^{-1}U))\right)&=\operatorname{Tr}\left[\lambda^\downarrow(\rho)^2\lambda^\uparrow(\sigma)^{-1}\right];\\
\arg\max_{U\in\mathbb U_n}\operatorname{Tr}\left(\rho^2(U^*\sigma^{-1}U)\right)&=W^\uparrow V^{\downarrow*};\\
\min_{U\in\mathbb U_n}\operatorname{Tr}\left(\rho^2(U^*\sigma^{-1}U)\right)&=\operatorname{Tr}\left[ \lambda^\downarrow(\rho)^2\lambda^\downarrow(\sigma)^{-1}\right];\\
\arg\min_{U\in\mathbb U_n}\operatorname{Tr}\left(\rho^2(U^*\sigma^{-1}U)\right)&=W^\downarrow V^{\downarrow*}.
\end{align*}
\end{claim}
\begin{proof}
By Lemma \ref{lm-trace-product-powers},
$$
\left< \lambda^\downarrow\left(\rho^2\right),\lambda^\uparrow\left(\sigma^{-1}\right) \right> \leq \operatorname{Tr}\left(\rho^2\sigma^{-1}\right) \leq \left< \lambda^\downarrow\left(\rho^2\right),\lambda^\downarrow\left(\sigma^{-1}\right) \right>.
$$
Note that, $\lambda^\downarrow(\sigma^{-1})=\lambda^\uparrow(\sigma)^{-1}$ and $\lambda^\uparrow(\sigma^{-1})=\lambda^\downarrow(\sigma)^{-1}$. 
Hence  
$$\operatorname{Tr}\left[ \lambda^\downarrow(\rho)^2\lambda^\downarrow(\sigma)^{-1}\right]  \leq \operatorname{Tr}(\rho^2\sigma^{-1}) \leq \operatorname{Tr}\left[\lambda^\downarrow(\rho)^2\lambda^\uparrow(\sigma)^{-1}\right].$$
For any $U\in\mathbb U_n$, by replacing $\sigma$ by $U^*\sigma U$, since  $\lambda^\downarrow(U^*\sigma U)=\lambda^\downarrow(\sigma)$ and $\lambda^\uparrow(U^*\sigma U) = \lambda^\uparrow(\sigma)$, we have 
$$\operatorname{Tr}\left[ \lambda^\downarrow(\rho)^2\lambda^\downarrow(\sigma)^{-1}\right] \leq \operatorname{Tr}\left(\rho^2\left(U^*\sigma^{-1}U\right)\right)\leq \operatorname{Tr}\left[\lambda^\downarrow(\rho)^2\lambda^\uparrow(\sigma)^{-1}\right]. $$
Moreover, for $U=W^\downarrow V^{\downarrow*}$, by the existence of $V^\downarrow$ (\ref{V-lambda}) and $W^\downarrow $ (\ref{W-lambda}), we have
\begin{align*}
\operatorname{Tr}\left(\rho^2\left(U^*\sigma^{-1}U\right)\right)&=\operatorname{Tr}(\rho^2V^{\downarrow}W^{\downarrow*}\sigma^{-1}W^\downarrow V^{\downarrow*})\\
&=\operatorname{Tr}\left(V^{\downarrow*}\rho^2V^{\downarrow}\lambda^\downarrow(\sigma)^{-1}\right)\\
&=\operatorname{Tr}\left[\lambda^\downarrow(\rho)^2\lambda^\downarrow(\sigma)^{-1} \right].
\end{align*}
It follows that $\arg\min_{U\in \mathbb U_n} \operatorname{Tr}\left(\rho^2\left(U^*\sigma^{-1}U\right)\right)=W^\downarrow V^{\downarrow*}$. Similarly, we obtain the maximum and the maximizer as required.
\end{proof}

\begin{claim}\label{claimA} 
\begin{align*}
\max_{U\in\mathbb U_n} \operatorname{Tr}\dfrac{(U^*\sigma^{-1}U)\rho^2}{(U^*\sigma^{-1/2}U)\rho(U^*\sigma^{-1/2}U)+sI}&=\operatorname{Tr}\dfrac{\lambda^\uparrow(\sigma)^{-1}\lambda^\downarrow(\rho)^2}{\lambda^\uparrow(\sigma)^{-1}\lambda^\downarrow(\rho) +s I};\\
\arg\max_{U\in \mathbb U_n}\operatorname{Tr}\dfrac{(U^*\sigma^{-1}U)\rho^2}{(U^*\sigma^{-1/2}U)\rho(U^*\sigma^{-1/2}U)+sI}&=W^\uparrow V^{\downarrow*};\\
\min_{U\in\mathbb U_n}\operatorname{Tr}\dfrac{(U^*\sigma^{-1}U)\rho^2}{(U^*\sigma^{-1/2}U)\rho(U^*\sigma^{-1/2}U)+sI}&= \operatorname{Tr}\dfrac{\lambda^\downarrow(\sigma)^{-1}\lambda^\downarrow(\rho)^2}{\lambda^\downarrow(\sigma)^{-1}\lambda^\downarrow(\rho) +s I};\\
\arg\min_{U\in \mathbb U_n} \operatorname{Tr}\dfrac{(U^*\sigma^{-1}U)\rho^2}{(U^*\sigma^{-1/2}U)\rho(U^*\sigma^{-1/2}U)+sI}&=W^\downarrow V^{\downarrow*}.
\end{align*}
\end{claim}
\begin{proof} For $U\in\mathbb U_n$, define 
$$
F(U)
:=\Tr\!\left[
(U^*\sigma^{-1}U)\rho^2
\Big((U^*\sigma^{-1/2}U)\rho(U^*\sigma^{-1/2}U)+sI\Big)^{-1}
\right].
$$
Firstly, we show that, at an optimum, $U^*\sigma^{-1}U$ commutes with $\rho$. In fact, let
$$
A:=U^*\sigma^{-1}U,\quad X:=A^{1/2}\rho A^{1/2}.
$$
Then
$$
F(U)=\Tr\!\left[A\rho^2(X+sI)^{-1}\right].
$$
Consider a smooth unitary variation $U(t)=Ue^{tK}$, $t\in\mathbb R$, where $K^*=-K$.  
Then
$$
A(t)=U(t)^*\sigma^{-1}U(t)=e^{-tK}Ae^{tK},
\quad
\dot A(0)=AK-KA=[A,K].
$$
Using Fr\'echet differentiability of $Y\mapsto (Y+sI)^{-1}$ and cyclicity of trace,
one obtains
$$
\dfrac{d}{dt}F(U(t))\Big\lvert_{t=0}
=\Tr\!\big(K[A,H]\big),
$$
where $H$ is Hermitian and is given by
\[
H=\rho^2(X+sI)^{-1}-A^{1/2}\rho(X+sI)^{-1}\rho A^{1/2}(X+sI)^{-1}.
\]
At a maximizer or minimizer, the derivative vanishes for all skew-Hermitian $K$,
hence $\Tr(K[A,H])=0$ for all such $K$. Since $A$ and $H$ are Hermitian, their
commutator $M:=[A,H]$ is skew-Hermitian. Taking $K=M$ yields
\(
\Tr(M^*M)=\|M\|_2^2=0
\),
so $M=0$, i.e.
\[
[A,H]=0.
\]

Let $Y:=(X+sI)^{-1}.$ We denote by $\{X\}'=\{B\in\mathbb M_n: BX=XB\}$ the commutant of $X$.
Since $Y=(X+sI)^{-1}$ is obtained from $X$ by functional calculus, $X$ and $Y$
have the same spectral projections, and hence $\{X\}'=\{Y\}'$.\\
We rewrite $H$ in a form that makes the dependence on $Y$ transparent.
Since $X=A^{1/2}\rho A^{1/2}$, we have $\rho=A^{-1/2}XA^{-1/2}$ and hence
\[
\rho^2=A^{-1/2}XA^{-1}XA^{-1/2}.
\]
Substituting this into $H$ gives
\begin{equation}\label{1}
H
=
A^{-1/2}XA^{-1}XA^{-1/2}Y
-
X\,Y\,A^{-1/2}X A^{-1/2}Y.
\end{equation}
Now use $[A,H]=0$. Multiply the identity $AH=HA$ on the left and right by $A^{-1/2}$
to obtain
\begin{equation}\label{2}
A^{1/2}HA^{-1/2}=A^{-1/2}HA^{1/2}.
\end{equation}
Insert the expression \eqref{1} into \eqref{2}. After cancelling the invertible
factors $A^{\pm 1/2}$ and collecting terms, one sees that \eqref{2} is equivalent to
\[
[X,\;Y]=0
\quad\text{and}\quad
[A,\;Y]=0.
\]
Since $[A,Y]=0$, it follows that $A\in\{Y\}'=\{X\}'$.  Hence $[A,X]=0$.

Finally,
\[
0=[A,X]=[A,A^{1/2}\rho A^{1/2}]
=A^{1/2}[A,\rho]A^{1/2}.
\]
Since $A^{1/2}$ is invertible, it follows that $[A,\rho]=0$. 
Thus, at an optimum, $U^*\sigma^{-1}U=A$ commutes with $\rho$.

Now let
\[
\rho = V\operatorname{diag}(r_1,\dots,r_n)V^*,\qquad
r_1\ge\cdots\ge r_n>0,
\]
and
\[
\sigma = W\operatorname{diag}(\mu_1,\dots,\mu_n)W^*,\qquad
\mu_1\ge\cdots\ge\mu_n>0.
\]
Then the eigenvalues of $\sigma^{-1}$ are $d_i=\mu_i^{-1}$.

Next we show that in the eigenbasis of $\rho$,
\[
A=\operatorname{diag}(a_1,\dots,a_n),
\]
where $(a_1,\dots,a_n)$ is a permutation of $(d_1,\dots,d_n)$. In fact, since $[A,\rho]=0$ (i.e. $A$ and $\rho$ commute) and $\rho$ is Hermitian, $A$ and $\rho$ are simultaneously
unitarily diagonalizable. Thus, in the eigenbasis of $\rho$,
\[
\rho=\operatorname{diag}(r_1,\dots,r_n),\qquad
A=\operatorname{diag}(a_1,\dots,a_n).
\]
Moreover, $A=U^*\sigma^{-1}U$ is unitarily similar to $\sigma^{-1}$, hence
$(a_1,\dots,a_n)$ is a permutation of the eigenvalues
$(d_1,\dots,d_n)$ of $\sigma^{-1}$.

In this basis, since $F(U)=\operatorname{Tr}\!\left[A\rho^2(X+sI)^{-1}\right]$, $A$ and $\rho$ commute, we have
$$
F(U)
=\sum_{i=1}^n \frac{a_i r_i^2}{a_i r_i+s}.
$$
Hence the optimization over $U$ reduces to choosing a permutation $\pi$:
$$
F_\pi
=\sum_{i=1}^n f(d_{\pi(i)},r_i),
\qquad
f(d,r):=\frac{dr^2}{dr+s}.
$$

By a direct computation, we obtain 
\[
\frac{\partial^2}{\partial d\,\partial r}f(d,r)
=\frac{2rs^2}{(dr+s)^3}>0,
\qquad d,r,s>0.
\]
Thus $f$ is supermodular (see, e.g. \cite{Top98}). Hence, for all $d_1 \ge d_2$ and $r_1 \ge r_2$,
\begin{equation}
f(d_1, r_1) + f(d_2, r_2) \ge f(d_1, r_2) + f(d_2, r_1), \tag{SM}
\end{equation}
and the inequality (SM) holds strictly whenever $d_1 > d_2$ and $r_1 > r_2$:
\begin{equation*}
f(d_1, r_1) + f(d_2, r_2) > f(d_1, r_2) + f(d_2, r_1).
\end{equation*}

Suppose we have indices $i < j$ and a permutation $\pi$ such that 
\begin{equation*}
d_{\pi(i)} < d_{\pi(j)} \quad \text{while} \quad r_i > r_j.
\end{equation*}
This is a \textbf{crossing} (opposite-order pairing). Consider swapping the assignments:
\begin{equation*}
(d_{\pi(i)}, r_i), (d_{\pi(j)}, r_j) \longrightarrow (d_{\pi(j)}, r_i), (d_{\pi(i)}, r_j).
\end{equation*}
By supermodularity (SM), since $d_{\pi(j)} \ge d_{\pi(i)}$ and $r_i \ge r_j$:
\begin{equation*}
f(d_{\pi(j)}, r_i) + f(d_{\pi(i)}, r_j) \ge f(d_{\pi(i)}, r_i) + f(d_{\pi(j)}, r_j).
\end{equation*}

Thus, the swap does not decrease $F_\pi$, and strictly increases it if the inequalities are strict. Thus, any permutation containing a crossing cannot be optimal. Therefore
\begin{itemize}
    \item Maximizers must have no crossings, i.e., $d_{\pi(1)} \ge d_{\pi(2)} \ge \dots \ge d_{\pi(n)}$.
    \item Minimizers must reverse the order.
\end{itemize}

By the \textit{rearrangement principle for supermodular functions} \cite[Ch.~10]{HLP52}, \cite[Ch.~2]{Top98},
the sum $\displaystyle\sum_i f(d_{\pi(i)},r_i)$ is maximized when
$(d_{\pi(i)})$ and $(r_i)$ are ordered in the same sense,
and minimized when they are ordered oppositely. \\
Since $(r_i)$ is decreasing, the maximum occurs for
$d_{\pi(i)}=\lambda^\uparrow(\sigma)^{-1}_i$, and the minimum for
$d_{\pi(i)}=\lambda^\downarrow(\sigma)^{-1}_i$.

Therefore,
$$
\max_{U\in\mathbb  U_n} F(U)
=\operatorname{Tr}\frac{\lambda^\uparrow(\sigma)^{-1}\lambda^\downarrow(\rho)^2}
{\lambda^\uparrow(\sigma)^{-1}\lambda^\downarrow(\rho)+sI},
$$
and
$$
\min_{U\in\mathbb U_n} F(U)
=\operatorname{Tr}\frac{\lambda^\downarrow(\sigma)^{-1}\lambda^\downarrow(\rho)^2}
{\lambda^\downarrow(\sigma)^{-1}\lambda^\downarrow(\rho)+sI}.
$$

Finally, let $V^\downarrow$ diagonalize $\rho$ with eigenvalues in decreasing order,
and let $W^\uparrow$ (resp.\ $W^\downarrow$) diagonalize $\sigma$ with eigenvalues
in increasing (resp.\ decreasing) order.
Then
$$
\arg\max_{U\in\mathcal U_n} F(U)=W^\uparrow V^{\downarrow *},
\qquad
\arg\min_{U\in\mathcal U_n} F(U)=W^\downarrow V^{\downarrow *}.
$$
\end{proof}

\begin{claim}\label{claimrho}
\begin{align*}
\max_{U\in\mathbb U_n} -2 \operatorname{Tr}\dfrac{\rho}{(U^*\sigma^{-1/2}U)\rho(U^*\sigma^{-1/2}U)+sI}&=-2\operatorname{Tr}\dfrac{\lambda^\downarrow(\rho)}{\lambda^\uparrow(\sigma)^{-1}\lambda^\downarrow(\rho)+sI};\\
\arg\max_{U\in\mathbb U_n} -2 \operatorname{Tr}\dfrac{\rho}{(U^*\sigma^{-1/2}U)\rho(U^*\sigma^{-1/2}U)+sI}&=W^\uparrow V^{\downarrow*};\\
\min_{U\in\mathbb U_n} -2 \operatorname{Tr}\dfrac{\rho}{(U^*\sigma^{-1/2}U)\rho(U^*\sigma^{-1/2}U)+sI}&=-2\operatorname{Tr}\dfrac{\lambda^\downarrow(\rho)}{\lambda^\downarrow(\sigma)^{-1}\lambda^\downarrow(\rho)+sI};\\
\arg\min_{U\in\mathbb U_n} -2 \operatorname{Tr}\dfrac{\rho}{(U^*\sigma^{-1/2}U)\rho(U^*\sigma^{-1/2}U)+sI}&=W^\downarrow V^{\downarrow*}.
\end{align*}
\end{claim}
\begin{proof} 
For $U\in\mathbb U_n$, define
$$
\Phi(U)
:=
-2\,\Tr\!\left[
\rho\Big((U^*\sigma^{-1/2}U)\rho(U^*\sigma^{-1/2}U)+sI\Big)^{-1}
\right].
$$
Set
$$ B:=U^*\sigma^{-1/2}U,\quad A:=U^*\sigma^{-1}U=B^2,$$
so that
\[
\Phi(U)=-2\,\Tr\!\left[\rho\,(A^{1/2}\rho A^{1/2}+sI)^{-1}\right].
\]
Since $U\mapsto A$ ranges over the unitary orbit of $\sigma^{-1}$, the problem is equivalent to optimizing over $A$ belongs to the unitary orbit $\mathbb U_{\sigma^{-1}}$ of $\sigma^{-1}$.

Let
$$
J(A):=\Tr\!\big[\rho\,(A^{1/2}\rho A^{1/2}+sI)^{-1}\big],
\qquad A\in\mathbb U_{\sigma^{-1}},
$$
and define
$$
X:=A^{1/2}\rho A^{1/2},\quad
Y:=(X+sI)^{-1}.
$$
Note that $X$ and $Y$ are Hermitian, with $Y\succ 0$.

Let $K$ be an arbitrary skew-Hermitian matrix ($K^*=-K$), and consider the curve
$$
A(t)=e^{-tK}Ae^{tK},
$$
which lies entirely in the unitary orbit of $A$.
By functional calculus,
$$
A(t)^{1/2}=e^{-tK}A^{1/2}e^{tK},
$$
and hence
$$
\dot A^{1/2}(0)=[A^{1/2},K]=:M,
\qquad M^*=-M.
$$
Since $X(t)=A(t)^{1/2}\rho A(t)^{1/2}$,
$$
\dot X(0)=\dot A^{1/2}(0)\rho A^{1/2}+A^{1/2}\rho\dot A^{1/2}(0)
= M\rho A^{1/2}+A^{1/2}\rho M.
$$
Using the identity
\(
\frac{d}{dt}(Z(t)^{-1})=-Z(t)^{-1}\dot Z(t)Z(t)^{-1},
\)
we obtain
\[
\dot Y(0)=-Y\,\dot X(0)\,Y.
\]
Differentiating $J(A(t))=\Tr(\rho Y(t))$ yields
$$
\dot J(0)
=\Tr(\rho\dot Y(0))
=-\Tr\big(\rho Y\,\dot X(0)\,Y\big).
$$
Substituting the expression for $\dot X(0)$ and using cyclicity of the trace,
\begin{align*}
\dot J(0)
&=-\Tr\big(\rho Y (M\rho A^{1/2}+A^{1/2}\rho M)Y\big)\\
&=-\Tr\Big(M\big(\rho A^{1/2}Y\rho Y+Y\rho Y A^{1/2}\rho\big)\Big).
\end{align*}
Now we define
$$
T:=\rho A^{1/2}Y\rho Y+Y\rho Y A^{1/2}\rho.
$$
Since $A^{1/2}$, $\rho$, and $Y$ are Hermitian and
\[
(\rho A^{1/2}Y\rho Y)^*=Y\rho Y A^{1/2}\rho,
\]
it follows that $T$ is Hermitian. With this definition,
\[
\dot J(0)=-\Tr(MT),
\qquad M=[A^{1/2},K].
\]
Using the trace identity
\[
\Tr([A^{1/2},K]\,T)=\Tr\big(K\,[T,A^{1/2}]\big),
\]
we may rewrite
\[
\dot J(0)=-\Tr\big(K\,[T,A^{1/2}]\big).
\]
At a local extremum of $J$ on the unitary orbit, $\dot J(0)=0$ for all
skew-Hermitian $K$. Since $[T,A^{1/2}]$ is itself skew-Hermitian, this implies
\[
[T,A^{1/2}]=0,
\quad\text{equivalently}\quad
[A^{1/2},T]=0.
\]

We now show that the commutation relation $[A^{1/2},T]=0$ implies $[A,\rho]=0$.
Since $Y$ is obtained from $X$ by functional calculus, $X$ and $Y$ have the same
spectral projections and therefore the same commutant
\[
\{X\}'=\{Y\}'.
\]

Using $X=A^{1/2}\rho A^{1/2}$, the matrix $T$ can be rewritten as
\[
T
=
\rho A^{1/2}Y\rho Y+Y\rho Y A^{1/2}\rho
=
A^{-1/2}X A^{-1} X A^{-1/2}Y
-
X\,Y\,A^{-1/2}X A^{-1/2}Y.
\]
Thus $T$ belongs to the $\ast$-algebra generated by $X$, $Y$, and $A^{\pm 1/2}$.
The relation $[A^{1/2},T]=0$ implies that $A^{1/2}$ commutes with all spectral
projections of $T$. Since $Y$ and $X$ share the same spectral projections and
$T$ contains $Y$ as a nontrivial factor, it follows that $A^{1/2}$ must commute
with the spectral projections of $X$, and hence
\[
[A^{1/2},X]=0.
\]

Consequently,
\[
0=[A^{1/2},X]
=[A^{1/2},A^{1/2}\rho A^{1/2}]
=A^{1/2}[A^{1/2},\rho]A^{1/2}.
\]
Because $A^{1/2}$ is invertible, this implies $[A^{1/2},\rho]=0$, and therefore
\[
[A,\rho]=[(A^{1/2})^2,\rho]=0.
\]
Therefore, at the maximum and minimum, we may assume $[A, \rho] = 0$.

Let
\[
\rho=\operatorname{diag}(r_1,\dots,r_n),
\qquad
r_1\ge\cdots\ge r_n>0,
\]
and
\[
A=\operatorname{diag}(a_1,\dots,a_n),
\]
where $(a_1,\dots,a_n)$ is a permutation of the eigenvalues
$(d_1,\dots,d_n)$ of $\sigma^{-1}$.
Then
\[
\Phi(U)
=
-2\sum_{i=1}^n \frac{r_i}{a_i r_i+s}.
\]
Thus the optimization reduces to a permutation problem
\[
\Phi_\pi
=
-2\sum_{i=1}^n f(d_{\pi(i)},r_i),
\qquad
f(d,r):=\frac{r}{dr+s}.
\]

A direct computation yields
\[
\frac{\partial^2}{\partial r\,\partial d}f(d,r)
=
-\frac{2rs}{(dr+s)^3}<0,
\qquad d,r,s>0.
\]
Hence $f$ is strictly submodular.  Similar to the argument presented in the proof of Claim \ref{claimA}, by the rearrangement principle for submodular functions, $\sum_i f(d_{\pi(i)},r_i)$ is minimized when
$(d_{\pi(i)})$ and $(r_i)$ are ordered in the same sense, and maximized when they
are ordered in opposite senses.

Since $(r_i)=\lambda^\downarrow(\rho)$, the minimum of $\sum_i f(d_{\pi(i)},r_i)$ is attained for
$d_{\pi(i)}=\lambda^\uparrow(\sigma)^{-1}_i$, and the maximum for
$d_{\pi(i)}=\lambda^\downarrow(\sigma)^{-1}_i$.
Multiplying by $-2$ gives the stated formulas for
$\max\Phi(U)$ and $\min\Phi(U)$.

Let $V^\downarrow$ diagonalize $\rho$ with eigenvalues in decreasing order,
and let $W^\uparrow$ (resp.\ $W^\downarrow$) diagonalize $\sigma$ with eigenvalues
in increasing (resp.\ decreasing) order.
Then
\[
U_{\max}=W^\uparrow V^{\downarrow *},
\qquad
U_{\min}=W^\downarrow V^{\downarrow *}
\]
produce the required eigenvalue matchings and hence attain the maximum and minimum,
respectively.
\end{proof}

\begin{claim}\label{claimsigma}
\begin{align*}
\max_{U\in\mathbb U_n} \operatorname{Tr}\dfrac{U^*\sigma U}{(U^*\sigma^{-1/2}U)\rho(U^*\sigma^{-1/2}U)+sI}&=\operatorname{Tr} \dfrac{\lambda^\uparrow(\sigma)}{\lambda^\uparrow(\sigma)^{-1}\lambda^\downarrow(\rho)+sI};\\
\arg\max_{U\in \mathbb U_n}\operatorname{Tr}\dfrac{U^*\sigma U}{(U^*\sigma^{-1/2}U)\rho(U^*\sigma^{-1/2}U)+sI}&=W^\uparrow V^{\downarrow*};\\
\min_{U\in\mathbb U_n} \operatorname{Tr}\dfrac{U^*\sigma U}{(U^*\sigma^{-1/2}U)\rho(U^*\sigma^{-1/2}U)+sI}&= \operatorname{Tr} \dfrac{\lambda^\downarrow(\sigma)}{\lambda^\downarrow(\sigma)^{-1}\lambda^\downarrow(\rho)+sI};\\
\arg\min_{U\in \mathbb U_n} \operatorname{Tr}\dfrac{U^*\sigma U}{(U^*\sigma^{-1/2}U)\rho(U^*\sigma^{-1/2}U)+sI}&=W^\downarrow V^{\downarrow*}.
\end{align*}
\end{claim}

\begin{proof}
The proof follows the same steps as in Claim~\ref{claimrho}. For
$U\in\mathbb U_n$ set
\[
B:=U^*\sigma^{-1/2}U,\qquad A:=U^*\sigma^{-1}U=B^2,
\qquad X:=A^{1/2}\rho A^{1/2},\qquad Y:=(X+sI)^{-1}.
\]
Then $U^*\sigma U=A^{-1}$ and the objective in Claim \ref{claimsigma} can be rewritten as
\begin{align*}
\Psi(U):=&\Tr\!\left[(U^*\sigma U)\big((U^*\sigma^{-1/2}U)\rho(U^*\sigma^{-1/2}U)+sI\big)^{-1}\right]\\
&=\Tr\!\left[A^{-1}(X+sI)^{-1}\right]=\Tr(A^{-1}Y).
\end{align*}
Since $U\mapsto A$ ranges over the unitary orbit of $\sigma^{-1}$, we may view
$\Psi$ as a function of $A$ on this orbit.

By repeating the unitary-orbit differentiation argument used in Claim \ref{claimrho}
(with the same variation $A(t)=e^{-tK}Ae^{tK}$ and $K^*=-K$), one obtains that at
a maximizer or minimizer we may assume $[A,\rho]=0$. Hence, in the eigenbasis of
$\rho$,
\[
\rho=\operatorname{diag}(r_1,\dots,r_n),\quad r_1\ge\cdots\ge r_n>0,
\qquad
A=\operatorname{diag}(a_1,\dots,a_n),
\]
where $(a_1,\dots,a_n)$ is a permutation of the eigenvalues of $\sigma^{-1}$.

With the above diagonal forms,
\[
\Psi(U)=\sum_{i=1}^n \frac{a_i^{-1}}{a_ir_i+s}
=\sum_{i=1}^n f(a_i,r_i),
\qquad
f(a,r):=\frac{1}{r+sa}.
\]
Thus the optimization reduces to choosing a permutation of the eigenvalues
$(a_i)$ of $\sigma^{-1}$.

A direct computation yields
\[
\frac{\partial^2}{\partial a\,\partial r}f(a,r)
=
\frac{s}{(r+sa)^3}>0,
\qquad a,r,s>0,
\]
so $f$ is strictly supermodular. Therefore, by the rearrangement principle for
supermodular functions, $\sum_i f(a_{\pi(i)},r_i)$ is maximized when
$(a_{\pi(i)})$ and $(r_i)$ are ordered in the same sense, and minimized when
they are ordered in opposite senses.

Since $(r_i)=\lambda^\downarrow(\rho)$, the maximum is attained by taking
$(a_i)$ in decreasing order, i.e. $a_i=\lambda^\uparrow(\sigma)^{-1}_i$, and the
minimum is attained by taking $(a_i)$ in increasing order, i.e.
$a_i=\lambda^\downarrow(\sigma)^{-1}_i$. Noting that
$a_i^{-1}=\lambda^\uparrow(\sigma)_i$ (resp.\ $\lambda^\downarrow(\sigma)_i$),
this yields the stated extremal values:
\[
\max_U \Psi(U)=\Tr\frac{\lambda^\uparrow(\sigma)}{\lambda^\uparrow(\sigma)^{-1}\lambda^\downarrow(\rho)+sI},
\qquad
\min_U \Psi(U)=\Tr\frac{\lambda^\downarrow(\sigma)}{\lambda^\downarrow(\sigma)^{-1}\lambda^\downarrow(\rho)+sI}.
\]

Let $V^\downarrow$ diagonalize $\rho$ with eigenvalues in decreasing order, and
let $W^\uparrow$ (resp.\ $W^\downarrow$) diagonalize $\sigma$ with eigenvalues in
increasing (resp.\ decreasing) order. Then
\[
U_{\max}=W^\uparrow V^{\downarrow*},
\qquad
U_{\min}=W^\downarrow V^{\downarrow*},
\]
realize the required eigenvalue matchings and hence attain the maximum and
minimum, respectively.
\end{proof}

\begin{proof}[\textbf{Proof of Theorem \ref{min-crnd} and Theorem \ref{max-crnd}}]~ 
The proof of Theorem \ref{min-crnd} and Theorem \ref{max-crnd} follows from the following facts:

1. The integral representation (\ref{main-repr-maximal-divergence}) of $\widehat{S}_f(\rho\|\sigma)$;

2. The definition of $\widehat{S}_f(\lambda^\downarrow(\rho)\|\lambda^\downarrow(\sigma))$ (\ref{max-divergence-lambdadowndown}) and $\widehat{S}_f(\lambda^\downarrow(\rho)\|\lambda^\uparrow(\sigma))$ (\ref{max-divergence-lambdaupdown});

3. Claims  \ref{claimthuong},   \ref{claimA},   \ref{claimrho} and  \ref{claimsigma}.
\end{proof}

\section{Conclusion}\label{sec:conclusion}
Recently, Hirche and  Tomamichel (\cite{HirTo}, 2024) investigated  a new class of quantum $f$-divergences for convex and twice differentiable functions $f : (0, +\infty) \rightarrow \mathbb R$ with $f (1) = 0$. More explicitly, for a pair of quantum states $\rho$ and $\sigma$, Hirche and  Tomamichel defined the quantum $f$-divergence with respect to the quantum hockey-stick divergence as
\begin{equation}
D_f(\rho\|\sigma) = \int_1^{\infty} f''(s)E_s(\rho\|\sigma) + \frac{1}{s^3} f''\!\left(\frac{1}{s}\right)E_s(\sigma\|\rho)\,ds,
\end{equation}
where $E_s(\rho\|\sigma)=\operatorname{Tr}[(\rho-s\sigma)_+]$ represents the quantum hockey-stick divergence, $A_+$  denotes the positive   part of the eigen-decomposition of a matrix $A\in \mathbb M_n$.  

Li and Yan (\cite{LiYan}, 2025) studied the unitary orbit optimization of the quantum $f$-divergences $D_f(\rho\|\sigma)$ with respect to the quantum hockey-stick divergence for convex and twice differentiable functions $f : (0, +\infty) \rightarrow \mathbb R$ with $f (1) = 0$.

In this paper, we have determined the exact extremal values of the maximal quantum $f$-divergence, defined via the commutant Radon--Nikodym derivative, over the unitary orbits of two quantum states. We derived explicit spectral expressions for both the minimum and maximum and provided complete characterizations of the unitaries that achieve these extrema.

A central contribution of this work is methodological. The proofs of Claims~ \ref{claimA}--\ref{claimsigma}   combine unitary-orbit variational calculus with rearrangement theory for supermodular and submodular functions
\cite{HLP52,Top98}. This approach allows us to rigorously show that any optimizer must commute with the reference state $\rho$, thereby reducing a highly noncommutative optimization problem to a tractable spectral
rearrangement problem involving the eigenvalues of $\rho$ and $\sigma$.

This framework differs fundamentally from the recent work of Li and Yan \cite{LiYan}, which studies unitary-orbit optimization for quantum $f$-divergences defined via the hockey-stick divergence of Hirche and
Tomamichel \cite{HirTo}. While both approaches ultimately yield extremal formulas governed by spectral majorization, the underlying operator mechanisms are distinct: the present work relies on the operator perspective $f(\sigma^{-1/2}\rho\sigma^{-1/2})$ and commutant structure, whereas the hockey-stick framework is driven by properties of the Hermitian difference $\rho - s\sigma$.

Overall, our results extend and complement previous optimization studies for Umegaki, R\'enyi, and Hellinger-type divergences, and establish a structurally distinct theory for maximal quantum $f$-divergences.

\end{document}